\documentclass[a4paper,11pt]{article}

\usepackage[applemac]{inputenc}
\usepackage{enumerate}
\usepackage{lmodern}
\usepackage[T1]{fontenc}
\usepackage{verbatim}
\usepackage{textcomp}
\usepackage[english]{babel}
\usepackage[a4paper,vmargin={3.5cm,3.5cm},hmargin={2.5cm,2.5cm}]{geometry}
\usepackage[font=sf, labelfont={sf,bf}, margin=1cm]{caption}
\usepackage[pdftex]{hyperref}
\usepackage[pdftex]{color,graphicx}

\usepackage{amsmath,amsfonts,amssymb,amsthm,mathrsfs}

\usepackage{cleveref}
  \crefname{theorem}{Theorem}{Theorems}
  \crefname{lemma}{Lemma}{Lemmas}
  \crefname{remark}{Remark}{Remarks}
  \crefname{proposition}{Proposition}{Propositions}
  \crefname{definition}{Definition}{Definitions}
  \crefname{corollary}{Corollary}{Corollaries}
  \crefname{section}{Section}{Sections}
  \crefname{figure}{Figure}{Figures}

\usepackage{color}

\newtheorem{theorem}{Theorem}[]

\newtheorem{proposition}[theorem]{Proposition}
\newtheorem{lemma}[theorem]{Lemma}
\newtheorem{corollary}[theorem]{Corollary}
\theoremstyle{definition}

\def\t{{\mathcal T}}

\def\ve{\varepsilon}
\def\wt{\widetilde}
\def\wh{\widehat}
\def\bm{\mathbf{m}}

\def\rd{\mathrm{d}}

\def\R{{\mathbb R}}
\def\P{{\mathbb P}}

\def\N{{\mathbb N}}
\def\Z{{\mathbb Z}}

\def\w{\mathrm{w}}

\def\build#1_#2^#3{\mathrel{
\mathop{\kern 0pt#1}\limits_{#2}^{#3}}}

\title{Bessel processes, the Brownian snake and super-Brownian motion}
\author{Jean-Fran\c cois Le Gall\thanks{Universit\'e Paris-Sud}}
\date{\small June 2014}
\begin{document}

\maketitle

\begin{abstract}
We prove that, both for the Brownian snake and for super-Brownian motion in
dimension one, the historical path corresponding to the minimal spatial position is
a Bessel process of dimension $-5$. We also discuss a spine decomposition 
for the Brownian snake conditioned on the minimizing path.
\end{abstract}

\section{Introduction}

Marc Yor used to say that ``Bessel processes are everywhere''. 
Partly in collaboration with Jim Pitman \cite{PY1,PY2}, he wrote several important papers,
 which considerably improved our knowledge of Bessel
processes and of their numerous applications. A whole chapter of Marc Yor's celebrated book with Daniel Revuz \cite{RY} is devoted to Bessel
processes and their applications to Ray-Knight theorems.  As a matter of fact,
Bessel processes play a major role in the study of properties of Brownian motion, and,
in particular, the three-dimensional Bessel process  is a key ingredient of the famous Williams decomposition of the Brownian excursion at its maximum.
In the present work, we show that Bessel processes also arise in similar 
properties of super-Brownian motion and the Brownian snake. Informally, we
obtain that, both for the Brownian snake and for super-Brownian motion, the
(historical) path reaching the minimal spatial position is a Bessel process of
negative dimension. 

Let us describe our results in a more precise way. We write
$(W_s)_{s\geq 0}$ for the Brownian  snake whose spatial motion is one-dimensional
Brownian motion. Recall that $(W_s)_{s\geq 0}$ is a Markov process taking values in the
space of all finite paths in $\R$, and for every $s\geq 0$, write $\zeta_s$ for the lifetime of $W_s$. We 
let $\N_0$ stand for the $\sigma$-finite excursion measure of $(W_s)_{s\geq 0}$ away from the trivial path
with  initial point $0$ and zero lifetime (see Section 2 for the precise normalization of $\N_0$).
We let $W_*$ be the minimal spatial position visited by the paths $W_s$, $s\geq 0$.
Then the ``law'' of $W_*$ under $\N_0$
is given by
\begin{equation}
\label{lawmini}
\N_0(W_* \leq -a) = \frac{3}{2a^2},
\end{equation}
for every $a>0$ (see \cite[Section VI.1]{LGZ} or \cite[Lemma 2.1]{LGW}). Furthermore, it is known that, $\N_0$ a.e., there is a unique instant $s_{\bm}$
such that $W_*=W_{s_{\bm}}(\zeta_{s_{\bm}})$. Our first main result (Theorem \ref{lawminipath}) shows that,
conditionally on $W_*=-a$, the random path $a+W_{s_{\bm}}$ is a Bessel process of
dimension $d=-5$ started from $a$ and stopped upon hitting $0$. Because of the relations between
the Browian snake and super-Brownian motion, this easily implies a similar result for the
unique historical path of one-dimensional super-Brownian motion that attains the minimal
spatial value (Corollary \ref{super-minipath}). Our second result (Theorem \ref{decotheo}) provides a ``spine decomposition'' of
the Brownian snake under $\N_0$ given the minimizing path $W_{s_{\bm}}$. Roughly speaking, this
decomposition involves Poisson processes  of Brownian snake excursions branching off
the minimizing path, which are conditioned not to attain the minimal value $W_*$. See
Theorem \ref{decotheo} for a more precise statement. 

Our proofs depend on various properties of the Brownian snake, including its strong 
Markov property and the ``subtree decomposition'' of the Brownian snake (\cite[Lemma V.5]{LGZ}, see Lemma
\ref{subtree} below)
starting from an arbitrary finite path $\w$. We also use the explicit distribution of the Brownian snake
under $\N_0$ at its first hitting time of a negative level: If $b>0$ and $S_b$ is the first  hitting
time of $-b$ by the Brownian snake, the path $b+W_{S_b}$ is distributed under $\N_0(\cdot\mid S_b<\infty)$
as a Bessel process of dimension $d=-3$ started from $b$ and stopped upon hitting $0$
(see Lemma \ref{Law-hitting-snake} below). Another
key ingredient (Lemma \ref{abscont}) is a variant of the absolute continuity relations
between Bessel processes that were discovered by Yor \cite{lacet} and studied in a
more systematic way in the paper \cite{PY1} by Pitman and Yor. 

Let us briefly discuss connections between our results and earlier work. As a special case of a
famous time-reversal theorem due to Williams \cite[Theorem 2.5]{Wil} (see also Pitman and Yor
\cite[Section 3]{PY2}, and in particular the examples treated in subsection (3.5) of \cite{PY2}), the time-reversal of a Bessel process of
dimension $d=-5$ started from $a$ and stopped upon hitting $0$ is a Bessel process of
dimension $d=9$ started from $0$ and stopped at its last passage time at $a$ -- This property can also
be found in \cite[Exercise XI.1.23]{RY}. Our results 
are therefore related to the appearance of nine-dimensional Bessel processes in 
limit theorems derived in \cite{LGW} and \cite{LGM}. Note however that in contrast with 
\cite{LGW} and \cite{LGM}, Theorem \ref{lawminipath} gives an exact identity in distribution and not
an asymptotic result. As a general remark, Theorem \ref{decotheo} is related to a number of 
``spine decompositions'' for branching processes that have appeared in the literature in 
various contexts. We finally note that a strong motivation for the present work came
from 
the forthcoming paper \cite{CLG}, which uses Theorems \ref{lawminipath} and \ref{decotheo} to
provide a new construction of the random metric space called the Brownian plane \cite{CLG0} 
and to give a number of explicit calculations of distributions related to this object. 

The paper is organized as follows. Section 2 presents a few preliminary results about
Bessel processes and the Brownian snake. Section 3 contains the statement and proof 
of our main results Theorems \ref{lawminipath} and \ref{decotheo}. Finally Section 4 gives
our applications to super-Brownian motion, which are more or less straightforward 
consequences of the results of Section 3.

\section{Preliminaries}

\subsection{Bessel processes}

We will be interested in Bessel processes of negative index. We refer to \cite{PY1} for the theory 
of Bessel processes, and we content ourselves with a brief presentation limited to the cases
of interest in this work. 
We let $B=(B_t)_{t\geq 0}$ be a linear Brownian motion and
 for every $\alpha>0$, we will consider the nonnegative process $R^{(\alpha)}=(R^{(\alpha)}_t)_{t\geq 0}$
 that solves the stochastic differential equation
 \begin{equation}
\label{SDE}
\rd R^{(\alpha)}_t = \rd B_t  - \frac{\alpha}{R^{(\alpha)}_t}\,\rd t,
\end{equation}
with a given (nonnegative) initial condition. 
To be specific, we require that equation \eqref{SDE} holds up to the first hitting time of $0$ by $R^{(\alpha)}$,
$$T^{(\alpha)}:=\inf\{t\geq 0: R^{(\alpha)}_t=0\},$$ 
and that
$R^{(\alpha)}_t=0$ for $t\geq T^{(\alpha)}$. Note that uniqueness in law and pathwise uniqueness hold for \eqref{SDE}. 

In the standard terminology (see e.g.\cite[Section 2]{PY1}), the process $R^{(\alpha)}$ is a Bessel process of index 
$\nu=-\alpha-\frac{1}{2}$, or dimension $d=1-2\alpha$. 
We will be interested especially in the cases $\alpha=2$ ($d=-3$) and $\alpha=3$ ($d=-5$). 

For notational convenience, we will assume that, for every $r\geq 0$, there is a 
probability measure $P_r$ such that both the Brownian motion $B$
and the Bessel processes $R^{(\alpha)}$ start from $r$ under $P_r$. 

Let us fix $r>0$ and argue under the probability measure $P_r$.
Fix $\delta\in(0,r)$ and set
$$T^{(\alpha)}_\delta:=\inf\{t\geq 0: R^{(\alpha)}_t=\delta\},$$
and $$T_\delta:=\inf\{t\geq 0: B_t=\delta\}.$$
The following absolute continuity lemma  is very closely related to results of 
\cite{lacet} (Lemma 4.5) and \cite{PY1} (Proposition 2.1), but we provide a short proof for the sake of completeness.
If $E$ is a metric space, $C(\R_+,E)$ stands for the space of all continuous functions
from $\R_+$ into $E$, which is equipped with the topology of uniform convergence on every compact interval.

\begin{lemma}
\label{abscont}
For every nonnegative measurable function $F$ on $C(\R_+,\R_+)$,
$$E_r[ F((R^{(\alpha)}_{t\wedge T^{(\alpha)}_\delta})_{t\geq 0})] 
= \Big(\frac{r}{\delta}\Big)^\alpha\, E_r\Big[ F((B_{t\wedge T_\delta})_{t\geq 0})\,\exp\Big(-\frac{\alpha(1+\alpha)}{2} \int_0^{T_\delta} \frac{\rd s}{B_s^2}\Big)\Big].$$
\end{lemma}

\begin{proof}
Write $(\mathcal{F}_t)_{t\geq 0}$ for the (usual augmentation of the) filtration generated by $B$.
For every $t\geq 0$, set
$$M_t:= \Big(\frac{r}{B_{t\wedge T_\delta}}\Big)^\alpha\,\exp\Big(-\frac{\alpha(1+\alpha)}{2} \int_0^{t\wedge T_\delta} \frac{\rd s}{B_s^2}\Big).$$
An application of It\^o's formula shows that $(M_t)_{t\geq 0}$ is an $(\mathcal{F}_t)$-local martingale. Clearly, $(M_t)_{t\geq 0}$ is bounded by $(r/\delta)^\alpha$
and is thus a uniformly integrable martingale, which converges as $t\to\infty$ to
$$M_\infty= \Big(\frac{r}{\delta}\Big)^\alpha\,\exp\Big(-\frac{\alpha(1+\alpha)}{2} \int_0^{T_\delta} \frac{\rd s}{B_s^2}\Big).$$
We define a probability measure $Q$ absolutely continuous with respect to $P_r$ by setting $Q=M_\infty\cdot P_r$. An application
of Girsanov's theorem shows that the process
$$B_t +\alpha \int_0^{t\wedge T_\delta} \frac{\rd s}{B_s} $$
is an $(\mathcal{F}_t)$-Brownian motion under $Q$. It follows that the law of $(B_{t\wedge T_\delta})_{t\geq 0}$ under $Q$ coincides with the 
law of $(R^{(\alpha)}_{t\wedge T^{(\alpha)}_\delta})_{t\geq 0}$ under $P_r$. This gives the desired result. \end{proof}

The formula of the next lemma is probably known, but we could not find a reference.

\begin{lemma}
\label{calculBessel}
For every $r>0$ and $a>0$,
$$E_{r}\Big[\exp\Big(-3\int_0^{T^{(2)}} \rd t\,(a +R^{(2)}_t)^{-2}\Big)\Big]= 1-\Big(\frac{r}{r+a}\Big)^2.$$
\end{lemma}

\begin{proof} An application of It\^o's formula shows that
$$M_t:=\Big(1 - \Big(\frac{R^{(2)}_t}{R^{(2)}_t + a}\Big)^2\Big)\, \exp\Big(-3\int_0^{t\wedge T^{(2)}} \rd s\,(a + R^{(2)}_s)^{-2}\Big)$$
is a local martingale. Clearly, $M_t$ is bounded by $1$ and is thus a uniformly integrable martingale.
Writing $E_r[M_{T^{(2)}}]=E_r[M_0]$ yields the desired result.
\end{proof}

\noindent{\it Remark.} An alternative proof of the formula of Lemma \ref{calculBessel}
will follow from forthcoming calculations: just use formula \eqref{decotech0} below with $G=1$,
noting that the left-hand side of this formula is then equal to $\N_0(-b-\ve<W_*\leq -b)$, which is computed
using \eqref{lawmini}. So strictly speaking we do not need the preceding proof. Still it seems a bit
odd to use the Brownian snake to prove the identity of Lemma \ref{calculBessel}, which has
to do with Bessel processes only.

\subsection{The Brownian snake}

We refer to \cite{LGZ} for the general theory of the Brownian snake, and only give a short presentation here.
We write $\mathcal{W}$ for the set of all finite paths in $\R$. An element of $\mathcal{W}$ is
a continuous mapping $\w:[0,\zeta]\longrightarrow \R$, where $\zeta=\zeta_{(\w)}\geq 0$ depends
on $\w$ and is called the lifetime of $\w$. We write $\wh \w=\w(\zeta_{(\w)})$ for the
endpoint of $\w$. For $x\in\R$, we set $\mathcal{W}_x:=\{\w\in \mathcal{W}:\w(0)=x\}$. 
The trivial path $\w$ such that $\w(0)=x$ and $\zeta_{(\w)}=x$
is identified with the point $x$ of $\R$, so that we can view $\R$ as a subset of $\mathcal{W}$.
The space $\mathcal{W}$ is equipped with the distance
$$d(\w,\w')=|\zeta_{(\w)}-\zeta_{(\w')}| + \sup_{t\geq 0} |\w(t\wedge \zeta_{(\w)})-\w'(t\wedge \zeta_{(\w')})|.$$

The Brownian snake $(W_s)_{s\geq 0}$ is a continuous Markov process with values in $\mathcal{W}$. We will write 
$\zeta_s=\zeta_{(W_s)}$ for the lifetime process of $W_s$. The process $(\zeta_s)_{s\geq 0}$ evolves like a
reflecting Brownian motion in $\R_+$. Conditionally on $(\zeta_s)_{s\geq 0}$, the evolution of $(W_s)_{s\geq 0}$
can be described informally as follows: When $\zeta_s$ decreases, the path $W_s$ is shortened from its tip,
and, when $\zeta_s$ increases, the path $W_s$ is extended by adding ``little pieces of linear Brownian motion''
at its tip. See \cite[Chapter IV]{LGZ} for a more rigorous presentation.

It is convenient to assume that the Brownian snake is defined on the canonical space $C(\R_+,\mathcal{W})$,
in such a way that, for $\omega=(\omega_s)_{s\geq 0}\in C(\R_+,\mathcal{W})$, we have $W_s(\omega)=\omega_s$. 
The notation $\P_\w$ then stands for the law of the Brownian snake started from $\w$. 

For every $x\in\R$, the trivial path $x$ is a regular recurrent point for the Brownian snake, and so we can make 
sense of the excursion measure $\N_x$ away from $x$, which is a $\sigma$-finite measure on $C(\R_+,\mathcal{W})$.
Under $\N_x$, the process $(\zeta_s)_{s\geq 0}$ is distributed according to the It\^o measure of positive excursions
of linear Brownian motion, which is normalized so that, for every $\ve>0$, 
$$\N_x\Big(\sup_{s\geq 0} \zeta_s >\ve\Big) =\frac{1}{2\ve}.$$
We write $\sigma:=\sup\{s\geq 0: \zeta_s>0\}$ for the duration of the excursion under $\N_x$. In a way analogous to the classical property of the It\^o excursion measure \cite[Corollary XII.4.3]{RY}, $\N_x$ is invariant under time-reversal,
meaning that $(W_{(\sigma-s)\vee 0})_{s\geq 0}$ has the same distribution as $(W_s)_{s\geq 0}$ under 
$\N_x$ .

Recall the notation
$$W_*:=\inf_{0\leq s\leq \sigma} \wh W_s=\inf_{0\leq s\leq \sigma}\inf_{0\leq t\leq \zeta_s} W_s(t),$$
and formula \eqref{lawmini} determining the law of $W_*$ under $\N_0$. 
It is known (see e.g.~\cite[Proposition 2.5]{LGW}) that $\N_x$ a.e. there is a unique instant $s_{\bm}\in[0,\sigma]$ such that
$\wh W_{s_{\bm}}=W_*$. One of our main objectives is to determine the law of $W_{s_{\bm}}$. We start
with two important lemmas. 

Our first lemma is concerned with the Brownian snake started from $\P_\w$, for some fixed $\w\in\mathcal{W}$,  and considered up to the first hitting time of $0$ by the lifetime process, that is
$$\eta_0:=\inf\{s\geq 0 : \zeta_s=0\}.$$
Then the values of the Brownian snake between times $0$ and $\eta_0$ can be 
classified according to ``subtrees'' branching off the initial path $\w$. To make this precise, let
$(\alpha_i,\beta_i)$, $i\in I$ be the excursion intervals away from $0$ of the
process
$$\zeta_s - \min_{0\leq r\leq s}\zeta_r$$
before time $\eta_0$. In other words, the intervals $(\alpha_i,\beta_i)$ are the connected components
of the open set $\{s\in[0,\eta_0]: \zeta_s > \min_{0\leq r\leq s}\zeta_r\}$. Using the properties of the
Brownian snake, it is easy to verify that $\P_\w$ a.s. for every $i\in I$, $W_{\alpha_i}=W_{\beta_i}$ 
is just the restriction of $\w$ to $[0,\zeta_{\alpha_i}]$, and the paths
$W_s$, $s\in[\alpha_i,\beta_i]$ all coincide over the time interval $[0,\zeta_{\alpha_i}]$. In order to describe the behavior of these paths beyond time $\zeta_{\alpha_i}$ we introduce, 
for every
$i\in I$, the element $W^i=(W^i_s)_{s\geq 0}$ of $C(\R_+,\mathcal{W})$ obtained by setting, for every $s\geq 0$,
$$W^i_s(t):=W_{(\alpha_i+s)\wedge \beta_i}(\zeta_{\alpha_i}+t)\;,\quad 0\leq t\leq \zeta^i_s:=\zeta_{(\alpha_i+s)\wedge \beta_i}-\zeta_{\alpha_i}.$$

\begin{lemma}
\label{subtree}
Under $\P_\w$, the point measure
$$\sum_{i\in I} \delta_{(\zeta_{\alpha_i}, W^i)}(\rd t,\rd \omega)$$
is a Poisson point measure on $\R_+\times C(\R_+,\mathcal{W})$ with intensity
$$2\,\mathbf{1}_{[0,\zeta_{(\w)}]}(t)\,\rd t\,\N_{\w(t)}(\rd \omega).$$
\end{lemma}

We refer to \cite[Lemma V.5]{LGZ} for a proof of this lemma. Our second lemma deals with
the distribution of the Brownian snake under $\N_0$ at the first hitting time of a negative level.
For every $b>0$, we set
$$S_b:= \inf\{s\geq 0: \wh W_s=-b\}$$
with the usual convention $\inf\varnothing =\infty$. 

\begin{lemma}
\label{Law-hitting-snake}
The law of the random path $W_{S_b}$ under the probability measure $\N_0(\cdot\mid S_b<\infty)$
is the law of the process $(R^{(2)}_t-b)_{0\leq t\leq T^{(2)}}$ under $P_b$. 
\end{lemma}

This lemma can be obtained as a very special case of Theorem 4.6.2 in \cite{DLG}. Alternatively, the lemma is also a special case of 
Proposition 1.4 in \cite{Del}, which relied on explicit calculations of capacitary distributions for the Brownian snake found in \cite{LGFourier}. Let us briefly explain how the result follows from \cite{DLG}. For every $x>-b$, set
$$u_b(x):=\N_x(S_b<\infty)= \frac{3}{2(x+b)^2}$$
where the second equality is just \eqref{lawmini}. Following the comments
at the end of Section 4.6 in \cite{DLG}, we get that the law of
$W_{S_b}$ under the probability measure $\N_0(\cdot\mid S_b<\infty)$ is the distribution of the process $X$
solving the stochastic differential equation
$$\rd X_t= \rd B_t + \frac{u'_b}{u_b}(X_t) \,\rd t\;,\quad X_0=0,$$
and stopped at its first hitting time of $-b$. Since $ \frac{u'_b}{u_b}(x)=-\frac{2}{x+b}$ we obtain the
desired result.

\section{The main results}

Our first theorem identifies the law of the minimizing path $W_{s_{\bm}}$.

\begin{theorem} 
\label{lawminipath}
Let $a>0$. Under $\N_0$, the conditional distribution of $W_{s_{\bm}}$ knowing that $W_*=-a$ is the distribution of 
the process $(R^{(3)}_t-a)_{0\leq t\leq T^{(3)}}$, where $R^{(3)}$ is a Bessel process of dimension $-5$
started from $a$, and $T^{(3)}=\inf\{t\geq 0: R^{(3)}_t=0\}$. 
\end{theorem}

In an integral form, the statement of the theorem means that, for any nonnegative measurable function $F$
on $\mathcal{W}_0$,
$$\N_0\big(F(W_{s_{\bm}})\big) = 3\int_0^\infty\frac{\rd a}{a^3} \,E_a\Big[ F\Big((R^{(3)}_t-a)_{0\leq t\leq T^{(3)}}\Big)\Big]$$
where we recall that the process $R^{(3)}$ starts from $a$ under $P_a$.

\begin{proof}
We fix three positive real numbers $\delta, K, K'$
such that $\delta<K<K'$, and we let $G$ be a bounded nonnegative continuous function on $\mathcal{W}_0$. For every
$\w\in\mathcal{W}_0$, we then set 
$$\tau_\delta(\w):=\inf\{t\geq 0: \w(t)=-\delta\}$$
and $F(\w):=G((\w(t))_{0\leq t\leq \tau_\delta(\w)})$ if $\tau_\delta(\w)<\infty$, $F(\w):=0$ otherwise. 

For every real $x$ and every integer $n\geq 1$, write $[x]_n$ for the largest real number of the form
$k2^{-n}$, $k\in\Z$, smaller than or equal to $x$. 
Using the special form of $F$ and the fact that $S_{[-W_*]_n}\uparrow s_{\bm}$ as $n\uparrow\infty$, $\N_0$ a.e.,  we easily get 
from the properties of the
Brownian snake that $F(W_{S_{[-W_*]_n}})= F(W_{S_*})$, for all $n$ large enough, $\N_0$ a.e.~on the event
$\{W_*<-\delta\}$.  By dominated convergence, we have then
\begin{align}
\label{decotech1}
&\N_0(F(W_{s_{\bm}})\mathbf{1}\{-K'\leq W_*\leq -K\})=\lim_{n\to\infty} \N_0(F(W_{S_{[-W_*]_n}}) \mathbf{1}\{K\leq [-W_*]_n\leq K'\})\nonumber\\
&=\lim_{n\to\infty} \sum_{K2^{n}\leq k\leq K'2^{n}} \N_0\Big(F(W_{S_{k2^{-n}}})\,\mathbf{1}\{S_{k2^{-n}}<\infty\}\,
\mathbf{1}\Big\{\min_{S_{k2^{-n}}\leq s\leq \sigma}\,{\wh W}_s > -(k+1)2^{-n}\Big\}\Big).
\end{align}

Let $b>\delta$ and $\varepsilon>0$. We use the strong Markov property of the Brownian snake at time $S_b$, 
together with Lemma \ref{subtree}, to get
\begin{align}
\label{decotech0}
& \N_0\Big(F(W_{S_{b}})\,\mathbf{1}\{S_{b}<\infty\}\,
\mathbf{1}\Big\{\min_{S_{b}\leq s\leq \sigma}\,{\wh W}_s > -b-\ve\Big\}\Big)\nonumber\\
&\quad= \N_0\Big(F(W_{S_{b}})\,\mathbf{1}\{S_{b}<\infty\}\,\exp\Big(-2\int_0^{\zeta_{S_b}} \rd t\,\N_{W_{S_b}(t)}(W_*>-b-\ve)\Big)\Big)\nonumber\\
&\quad= \N_0\Big(F(W_{S_{b}})\,\mathbf{1}\{S_{b}<\infty\}\,\exp\Big(-3\int_0^{\zeta_{S_b}} \rd t\,(b+\ve+W_{S_b}(t))^{-2}\Big)\Big)\nonumber\\
&\quad= \frac{3}{2b^2}\,E_b\Big[ F((R^{(2)}_t-b)_{0\leq t\leq T^{(2)}})\,\exp\Big(-3\int_0^{T^{(2)}} \rd t\,(\ve + R^{(2)}_t)^{-2}\Big)\Big]
\end{align}
using \eqref{lawmini} in the second equality, and Lemma \ref{Law-hitting-snake} and \eqref{lawmini} again in the third one. Recall the definition of the stopping times $T^{(\alpha)}_\delta$ before Lemma \ref{abscont}.
From the special form
of the function $F$, and then the strong Markov property of the process $R^{(2)}$ at time $T^{(2)}_{b-\delta}$, we obtain that
\begin{align}
\label{decotech00}
&E_b\Big[ F((R^{(2)}_t-b)_{0\leq t\leq T^{(2)}})\,\exp\Big(-3\int_0^{T^{(2)}} \rd t\,(\ve + R^{(2)}_t)^{-2}\Big)\Big]
\nonumber\\
&\quad= E_b\Big[ G((R^{(2)}_t-b)_{0\leq t\leq T^{(2)}_{b-\delta}})\,\,\exp\Big(-3\int_0^{T^{(2)}} \rd t\,(\ve + R^{(2)}_t)^{-2}\Big)\Big]\nonumber\\
&\quad = E_b\Big[ G((R^{(2)}_t-b)_{0\leq t\leq T^{(2)}_{b-\delta}})\,\,\exp\Big(-3\int_0^{T^{(2)}_{b-\delta}} \rd t\,(\ve + R^{(2)}_t)^{-2}\Big)\nonumber\\
&\qquad\qquad \qquad\times E_{b-\delta}\Big[\exp\Big(-3\int_0^{T^{(2)}} \rd t\,(\ve +R^{(2)}_t)^{-2}\Big)\Big]\Big].
\end{align}

Using the formula of Lemma \ref{calculBessel} and combining \eqref{decotech0}
and \eqref{decotech00}, we arrive at
\begin{align*}
&\N_0\Big(F(W_{S_{b}})\,\mathbf{1}\{S_{b}<\infty\}\,
\mathbf{1}\Big\{\min_{S_{b}\leq s\leq \sigma}\,{\wh W}_s > -b-\ve\Big\}\Big)\\
&= \frac{3}{2b^2}\Big(1-\Big(\frac{b-\delta}{b-\delta+\ve}\Big)^2\Big)\;
E_b\Big[ G((R^{(2)}_t-b)_{0\leq t\leq T^{(2)}_{b-\delta}})\,\,\exp\Big(-3\int_0^{T^{(2)}_{b-\delta}} \rd t\,(\ve + R^{(2)}_t)^{-2}\Big)\Big].
\end{align*}
Hence,
\begin{align*}
&\lim_{\ve\to 0} \ve^{-1} \N_0\Big(F(W_{S_{b}})\,\mathbf{1}\{S_{b}<\infty\}\,
\mathbf{1}\Big\{\min_{S_{b}\leq s\leq \sigma}\,{\wh W}_s > -b-\ve\Big\}\Big)\\
&\quad= \Big(\frac{3}{b^2(b-\delta)}\Big)\,E_b\Big[ G((R^{(2)}_t-b)_{0\leq t\leq T^{(2)}_{b-\delta}})\,\,\exp\Big(-3\int_0^{T^{(2)}_{b-\delta}} \rd t\,(R^{(2)}_t)^{-2}\Big)\Big].
\end{align*}
At this stage we use Lemma \ref{abscont} twice to see that
\begin{align*}
&E_b\Big[ G((R^{(2)}_t-b)_{0\leq t\leq T^{(2)}_{b-\delta}})\,\,\exp\Big(-3\int_0^{T^{(2)}_{b-\delta}} \rd t\,(R^{(2)}_t)^{-2}\Big)\Big]\\
&\quad= \Big(\frac{b}{b-\delta}\Big)^2\, E_b\Big[ G((B_t-b)_{0\leq t\leq T_{b-\delta}})\,\exp\Big(-6 \int_0^{T_{b-\delta}} \frac{\rd s}{B_s^2}\Big)\Big]\\
&\quad=  \Big(\frac{b}{b-\delta}\Big)^{-1}\,E_b\Big[ G((R^{(3)}_t-b)_{0\leq t\leq T^{(3)}_{b-\delta}})\Big]
\end{align*}
Summarizing, we have
$$\lim_{\ve\to 0} \ve^{-1} \N_0\Big(F(W_{S_{b}})\,\mathbf{1}\{S_{b}<\infty\}\,
\mathbf{1}\Big\{\min_{S_{b}\leq s\leq \sigma}\,{\wh W}_s > -b-\ve\Big\}\Big)
=\frac{3}{b^3}\,E_b\Big[ G((R^{(3)}_t-b)_{0\leq t\leq T^{(3)}_{b-\delta}})\Big].$$
Note that the right-hand side of the last display is a continuous function of 
$b\in(\delta,\infty)$. Furthermore, a close look at the preceding arguments shows that
the convergence is uniform when $b$ varies over an interval of the form
$[\delta',\infty)$, where $\delta'>\delta$. We can therefore return to 
\eqref{decotech1} and obtain that
\begin{align*}
&\N_0(F(W_{s_{\bm}})\mathbf{1}\{-K'\leq W_*\leq -K\})\\
&\quad=\lim_{n\to\infty}
\int_K^{K'} \rd b\,2^n\N_0\Big(F(W_{S_{[b]_n}})\,\mathbf{1}\{S_{[b]_n}<\infty\}\,
\mathbf{1}\Big\{\min_{S_{[b]_n}\leq s\leq \sigma}\,{\wh W}_s > -[b]_n-2^{-n}\Big\}\Big)\\
&\quad=3\,\int_K^{K'} \frac{\rd b}{b^3}\,E_b\Big[ G((R^{(3)}_t-b)_{0\leq t\leq T^{(3)}_{b-\delta}})\Big].
\end{align*}
The result of the theorem now follows easily. 

\end{proof}

We turn to a statement describing the structure of subtrees branching off the minimizing path $W_{s_{\bm}}$.
In a sense, this is similar to Lemma \ref{subtree} above (except that we will need to consider 
separately subtrees
branching {\it before} and {\it after} time $s_{\bm}$, in the time scale of the Brownian snake). 
Since 
$s_{\bm}$ is not a stopping time of the Brownian snake, it is of course impossible to use the strong Markov property in
order to apply Lemma \ref{subtree}. Still this lemma will play an important role.

We argue under the excursion measure $\N_0$ and, for every $s\geq 0$, we set
$$\hat \zeta_s:=\zeta_{(s_{\bm}+s)\wedge \sigma}\ ,\quad\check \zeta_s:=\zeta_{(s_{\bm}-s)\vee0}.$$
We let $(\hat a_i,\hat b_i)$, $i\in I$ be the excursion intervals of $\hat \zeta_s$ above its past minimum. Equivalently,
the intervals $(\hat a_i,\hat b_i)$, $i\in I$ are the connected components of 
the set 
$$\Big\{s\geq 0: \hat \zeta_s > \min_{0\leq r\leq s}\hat\zeta_r\Big\}.$$
Similarly, we let $(\check a_j,\check b_j)$, $j\in J$ be the excursion intervals of $\check\zeta_s$ above its past minimum.
We may assume that the indexing sets $I$ and $J$ are disjoint. In terms of the tree $\t_\zeta$ coded by
the excursion $(\zeta_s)_{0\leq s\leq \sigma}$ under $\N_0$ (see e.g. \cite[Section 2]{LGtree}),
each interval $(\hat a_i,\hat b_i)$ or $(\check a_j,\check b_j)$ corresponds to a subtree of $\t_\zeta$ branching off
the ancestral line of the vertex associated with $s_{\bm}$. We next consider the spatial displacements corresponding to these subtrees. 
For every $i\in I$, we let $W^{(i)}=(W^{(i)}_s)_{s\geq 0}\in C(\R_+,\mathcal{W})$ be defined by
$$W^{(i)}_s(t) = W_{s_{\bm}+(\hat a_i+s)\wedge\hat b_i} (\zeta_{s_{\bm}+\hat a_i} +t)\ , \quad 0\leq t\leq \zeta_{s_{\bm}+(\hat a_i+s)\wedge\hat b_i} -\zeta_{s_{\bm}+\hat a_i}.$$
Similarly, for every $j\in J$,
$$W^{(j)}_s(t) = W_{s_{\bm}-(\check a_j+s)\wedge \check b_j} (\zeta_{s_{\bm}-\check a_j} +t)\ , \quad 0\leq t\leq \zeta_{s_{\bm}-(\check a_j+s)\wedge\check b_j} -\zeta_{s_{\bm}-\check a_j}.$$
We finally introduce the point measures on $\R_+\times C(\R_+,\mathcal{W})$ defined by
$$\hat{\mathcal N} = \sum_{i\in I} \delta_{(\zeta_{s_{\bm}+\hat a_i}, W^{(i)})}\ ,\quad \check{\mathcal N} = \sum_{j\in J} \delta_{(\zeta_{s_{\bm}-\hat a_j}, W^{(j)})}.$$
If $\omega=(\omega_s)_{s\geq 0}$ belongs to $C(\R_+,\mathcal{W})$, we set $\omega_*:=\inf\{\omega_s(t):s\geq 0,0\leq t\leq \zeta_{(\omega_s)}\}$.

\begin{theorem}
\label{decotheo}
Under $\N_0$, conditionally on the minimizing path $W_{s_{\bm}}$, the point measures $\hat{\mathcal N}(\rd t,\rd \omega)$ and $ \check{\mathcal N}(\rd t,\rd \omega)$
are independent and their common conditional distribution is that of a Poisson point measure with intensity
$$2\,\mathbf{1}_{[0,\zeta_{s_{\bm}}]}(t)\,\mathbf{1}_{\{\omega_*> \wh W_{s_\bm}\}}\,\rd t\,\N_{W_{s_{\bm}}(t)}(\rd \omega).$$
\end{theorem}

Clearly, the constraint $\omega_*>  \wh W_{s_\bm}$ corresponds to the fact that none of the spatial positions in the subtrees
branching off the ancestral line of $p_\zeta(s_{\bm})$ can be smaller than $W_*= \wh W_{s_\bm}$, by the very definition of
$W_*$. 

\begin{proof}
We will first argue that the conditional distribution of $\hat{\mathcal N}$ given $W_{s_{\bm}}$ is as
described in the theorem. To this end, we fix again $\delta, K,K'$ such that $0<\delta<K<K'$, 
and we use the notation
$\tau_\delta(\w)$ introduced in the proof of Theorem \ref{lawminipath}. On the event where $W_*<-\delta$, we also set
$$\hat{\mathcal N}_\delta = \build{\sum_{i\in I}}_{\zeta_{s_{\bm}+\hat a_i}\leq \tau_\delta(W_{s_{\bm}})}^{} \delta_{(\zeta_{s_{\bm}+\hat a_i}, W^{(i)})}.$$
Informally, considering only the subtrees that occur after $s_{\bm}$ in the time scale of the 
Brownian snake, $\hat{\mathcal N}_\delta$ corresponds to those subtrees that branch off the minimizing
path $W_{s_{\bm}}$ before this path hits the level $-\delta$. 

Next, let $\Phi$ be a bounded nonnegative measurable function on the space of all point
measures on $\R_+\times C(\R_+,\mathcal{W})$ -- we should restrict this space to 
point measures satisfying appropriate $\sigma$-finiteness conditions, but we omit the details -- and let $\Psi$ be a bounded continuous function
on $C(\R_+,\mathcal{W})$. To simplify notation, we write 
$W_{\leq s_{\bm}}$ for the process $(W_{s\wedge s_{\bm}})_{s\geq 0}$ viewed as a random element
of  $C(\R_+,\mathcal{W})$, and we use the similar notation $W_{\leq S_b}$. For every $b>0$,
let the point measure $\hat{\mathcal N}_\delta^{(b)}$ be defined (only on the event
where $S_{b}<\infty$) in a way analogous to $\hat{\mathcal N}_\delta$ but replacing the path
$W_{s_{\bm}}$ with the path $W_{S_{b}}$: To be specific, $\hat{\mathcal N}_\delta^{(b)}$
accounts for those subtrees (occurring after $S_{b}$ in the time scale of the 
Brownian snake) that branch off $W_{S_{b}}$ before this path hits $-\delta$.

As in \eqref{decotech1}, we have then
\begin{align}
\label{decotech2}
\N_0\Big(\Psi(W_{\leq s_{\bm}})\mathbf{1}\{-K'\leq W_*\leq -K\}\,\Phi(\hat{\mathcal N}_\delta)\Big)
&=\lim_{n\to\infty} \sum_{K2^{n}\leq k\leq K'2^{n}} \N_0\Big(\Psi(W_{\leq S_{k2^{-n}}})\,\mathbf{1}\{S_{k2^{-n}}<\infty\}\,\nonumber\\
&\qquad \mathbf{1}\Big\{\min_{S_{k2^{-n}}\leq s\leq \sigma}\,{\wh W}_s > -(k+1)2^{-n}\Big\}\,\Phi(\hat{\mathcal N}_\delta^{(k2^{-n})})\Big).
\end{align}
The point
in \eqref{decotech2} is the fact that, $\N_0$ a.e., if $n$ is sufficiently large, and if $k\geq K2^{-n}$ is the
largest integer such that  $S_{k2^{-n}}<\infty$, the paths $W_{s_{\bm}}$ and $W_{S_{k2^{-n}}}$
are the same up to a time which is greater than $\tau_\delta(W_{s_{\bm}})$, and the point measures
$\hat{\mathcal N}_\delta$ and $\hat{\mathcal N}_\delta^{(k2^{-n})}$ coincide. 

Next fix $b>\delta$ and, for $\ve>0$, consider the quantity
\begin{equation}
\label{decotech3}
\N_0\Big(\Psi(W_{\leq S_{b}})\,\mathbf{1}\{S_{b}<\infty\}\,
\mathbf{1}\Big\{\min_{S_{b}\leq s\leq \sigma}\,{\wh W}_s > -b-\ve\Big\}\,\Phi(\hat{\mathcal N}_\delta^{(b)})\Big).
\end{equation}
To evaluate this quantity, we again apply the strong Markov property of the Brownian snake at time $S_b$. For notational
convenience, we suppose that,
on a certain probability space, we have a random point measure $\mathcal{M}$ on $\R_+\times C(\R_+,\mathcal{W})$ 
and, for every $\w\in\mathcal{W}_0$, a probability measure $\Pi_\w$ under which 
$\mathcal{M}(\rd t,\rd \omega)$ is Poisson with intensity
$$2\,\mathbf{1}_{[0,\zeta_\w]}(t)\,\rd t\,\N_{\w(t)}(\rd \omega).$$
By the strong Markov property at $S_b$ and Lemma \ref{subtree}, the quantity \eqref{decotech3} is equal to
$$\N_0\Big(\Psi(W_{\leq S_{b}})\,\mathbf{1}\{S_{b}<\infty\}\,
\Pi_{W_{S_b}}\Big(\mathbf{1}\{\mathcal{M}(\{(t,\omega): \omega_*\leq -b-\ve\})=0\}\,
\Phi(\mathcal{M}_{\leq\tau_\delta(W_{S_b})})\Big)\Big),$$
where $\mathcal{M}_{\leq\tau_\delta(W_{S_b})}$ denotes the restriction of the point measure $\mathcal{M}$
to $[0, \tau_\delta(W_{S_b})]\times C(\R_+,\mathcal{W})$. Write $W_{S_b}^{(\delta)}$ for the restriction of the
path $W_{S_b}$ to $[0,\tau_\delta(W_{S_b})]$. We have then
\begin{align*}
&\Pi_{W_{S_b}}\Big(\mathbf{1}\{\mathcal{M}(\{(t,\omega): \omega_*\leq -b-\ve\})=0\}\,
\Phi(\mathcal{M}_{\leq\tau_\delta(W_{S_b})})\Big)\\
&\ = \Pi_{W_{S_b}}(\mathcal{M}(\{(t,\omega): \omega_*\leq -b-\ve\})=0)\,
\Pi_{W_{S_b}}\Big(\Phi(\mathcal{M}_{\leq\tau_\delta(W_{S_b})})\,\Big|\,\mathcal{M}(\{(t,\omega): \omega_*\leq -b-\ve\})=0 \Big)\\
&\ = \Pi_{W_{S_b}}(\mathcal{M}(\{(t,\omega): \omega_*\leq -b-\ve\})=0)\, \Pi_{W_{S_b}^{(\delta)}}\Big(
\Phi(\mathcal{M})\,\Big|\,\mathcal{M}(\{(t,\omega): \omega_*\leq -b-\ve\})=0 \Big),
\end{align*}
using standard properties of Poisson measures in the last equality. Summarizing, we see that the
quantity \eqref{decotech3} coincides with
\begin{equation}
\label{decotech4}
\N_0\Big(\Psi(W_{\leq S_{b}})\,H(W_{S_{b}},b+\ve)\,\mathbf{1}\{S_{b}<\infty\}\,
\Pi_{W_{S_b}}(\mathcal{M}(\{(t,\omega): \omega_*\leq -b-\ve\})=0)\Big),
\end{equation}
where, for every $\w\in\mathcal{W}_0$
such that $\tau_\delta(\w)<\infty$, for every $a>\delta$, $H(\w,a):=\wt H((\w(t))_{0\leq t\leq \tau_\delta(\w)},a)$, and
 the
function $\wt H$ is given by
$$\wt H(\w,a):= \Pi_{\w}\Big(
\Phi(\mathcal{M})\,\Big|\,\mathcal{M}(\{(t,\omega): \omega_*\leq -a\})=0 \Big),$$
this definition making sense if $\w\in\mathcal{W}_0$ does not hit $-a$.
By the strong Markov property at $S_b$ and again Lemma \ref{subtree}, the quantity \eqref{decotech4} is also equal to
$$\N_0\Big(\Psi(W_{\leq S_{b}})\,H(W_{S_{b}},b+\ve)\,\mathbf{1}\{S_{b}<\infty\}\,
\mathbf{1}\Big\{\min_{S_{b}\leq s\leq \sigma}\,{\wh W}_s > -b-\ve\Big\}\Big).$$

We may now come back to \eqref{decotech2}, and get from the previous observations that
\begin{align*}
&\N_0\Big(\Psi(W_{\leq s_{\bm}})\mathbf{1}\{-K'\leq W_*\leq -K\}\,\Phi(\hat{\mathcal N}_\delta)\Big)\\
&\quad = \lim_{n\to\infty} \sum_{K2^{n}\leq k\leq K'2^{n}} 
\N_0\Big(\Psi(W_{\leq S_{k2^{-n}}})\,H(W_{S_{k2^{-n}}},(k+1)2^{-n})\,\\
&\qquad\qquad\qquad\qquad\qquad\qquad\mathbf{1}\{S_{k2^{-n}}<\infty\}\,
\mathbf{1}\Big\{\min_{S_{k2^{-n}}\leq s\leq \sigma}\,{\wh W}_s > -(k+1)2^{-n}\Big\}\Big)\\
&\quad=\lim_{n\to\infty} \N_0\Big(\Psi(W_{\leq S_{[-W_*]_n}})\,H(W_{S_{[-W_*]_n}},[-W_*]_n-2^{-n}) \,\mathbf{1}\{K\leq [-W_*]_n\leq K'\}\Big)\\
&\quad= \N_0\Big(\Psi(W_{\leq s_{\bm}})\,H(W_{s_{\bm}},-W_*)\,\mathbf{1}\{-K'\leq W_*\leq -K\}\Big).
\end{align*}
To verify the last equality, recall that the paths $W_{S_{[-W_*]_n}} $ and $W_{s_{\bm}}$
coincide up to their first hitting time of $-\delta$, for all $n$ large enough, $\N_0$ a.e., and use
also the fact that the function $H(\w,a)$ is Lipschitz in the variable $a$ on every compact
subset of $(\delta,\infty)$, uniformly in the variable $\w$. 

From the definition of $H$, we have then
\begin{align*}
&\N_0\Big(\Psi(W_{\leq s_{\bm}})\mathbf{1}\{-K'\leq W_*\leq -K\}\,\Phi(\hat{\mathcal N}_\delta)\Big)\\
&\quad= \N_0\Big(\Psi(W_{\leq s_{\bm}})\mathbf{1}\{-K'\leq W_*\leq -K\}\,
\Pi_{W_{s_{\bm}}^{(\delta)}}\Big(
\Phi(\mathcal{M})\,\Big|\,\mathcal{M}(\{(t,\omega): \omega_*\leq W_*\})=0 \Big)\Big),
\end{align*}
where $W_{s_{\bm}}^{(\delta)}$ denotes the restriction of $W_{s_{\bm}}$ to $[0,\tau_\delta(W_{s_{\bm}})]$.
From this, and since $W_*=\wh W_{s_\bm}$, we obtain that the condi\-tional distribution of 
$\hat{\mathcal N}_\delta$ given $W_{\leq s_{\bm}}$ is (on the event where $W_*<-\delta$) the law of 
a Poisson point measure with intensity
$$2\,\mathbf{1}_{[0,\tau_\delta(W_{s_{\bm}})]}(t)\,\mathbf{1}_{\{\omega_*>\wh W_{s_\bm}\}}\,\rd t\,\N_{W_{s_{\bm}}(t)}(\rd \omega).$$
 Since $\delta$ is arbitrary, it easily follows
that the conditional distribution of 
$\hat{\mathcal N}$ given $W_{\leq s_{\bm}}$ is that of a Poisson measure with intensity
$$2\,\mathbf{1}_{[0,\zeta_{W_{s_{\bm}}}]}(t)\,\mathbf{1}_{\{\omega_*>\wh W_{s_\bm}\}}\,\rd t\,\N_{W_{s_{\bm}}(t)}(\rd \omega).$$
Note that this conditional distribution only depends on $W_{s_{\bm}}$, meaning that 
$\hat{\mathcal N}$ is conditionally independent of $W_{\leq s_{\bm}}$ given $W_{s_{\bm}}$.

Since the measure $\N_0$ is invariant under time-reversal, we also get that the conditional distribution
of $\check{\mathcal N}$ given $W_{s_{\bm}}$ is the same as the conditional distribution of $\hat{\mathcal N}$
given $W_{s_{\bm}}$.
Finally, $\check{\mathcal N}$ is a measurable function of $W_{\leq s_{\bm}}$ and since 
 $\hat{\mathcal N}$ is conditionally independent of $W_{\leq s_{\bm}}$ given $W_{s_{\bm}}$, we get that
 $\check{\mathcal N}$ and $\hat{\mathcal N}$ are conditionally independent given $W_{s_{\bm}}$. 
\end{proof}

\section{Applications to super-Brownian motion}

We will now discuss applications of the preceding results to super-Brownian motion.
Let $\mu$ be a (nonzero) finite measure on $\R$. We denote the topological support of $\mu$
by ${\rm supp}(\mu)$ and always assume that 
$$m:=\inf\,{\rm supp}(\mu) > -\infty.$$
We then consider a super-Brownian motion $X=(X_t)_{t\geq 0}$ with quadratic 
branching mechanism $\psi(u)=2u^2$ started from $\mu$. The particular choice of the normalization
of $\psi$ is motivated by the connection with the Brownian snake. Let us 
recall this connection following Section IV.4 of \cite{LGZ}. We consider a Poisson 
point measure $\mathcal{P}(\rd x,\rd \omega)$ on $\R\times C(\R_+,\mathcal{W})$ with intensity
$$\mu(\rd x)\,\N_x(\rd \omega).$$
Write
$$\mathcal{P}(\rd x,\rd \omega)=\sum_{i\in I} \delta_{(x^i,\omega^i)}(\rd x,\rd \omega)$$
and for every $i\in I$, let $\zeta^i_s=\zeta_{(\omega^i_s)}$, $s\geq 0$, stand for the lifetime
process associated with $\omega^i$. Also, for every $r\geq 0$ and $s\geq 0$, let
$\ell^r_s(\zeta^i)$ be the local time at level $r$ and at time $s$ of the process 
$\zeta^i$. We may and will construct the super-Brownian motion $X$ by setting
$X_0=\mu$ and for every $r>0$, for every nonnegative measurable function $\varphi$ on $\R$,
\begin{equation}
\label{super-snake}
\langle X_r,\varphi\rangle
=\sum_{i\in I} \int_0^\infty \rd _s\ell^r_s(\zeta^i)\,\varphi(\wh\omega^i_s),
\end{equation}
where the notation $\rd _s\ell^r_s(\zeta^i)$ refers to integration with respect to
the increasing function $s\to\ell^r_s(\zeta^i)$. 

A major advantage of the Brownian snake construction is the fact that it also yields
an immediate definition of the historical super-Brownian motion $Y=(Y_r)_{r\geq 0}$ associated with
$X$ (we refer to \cite{DP} or \cite{Dyn} for the general theory of historical superprocesses). 
For every $r\geq 0$, $Y_r$ is a finite measure on the subset of $\mathcal{W}$
consisting of all stopped paths with lifetime $r$. We have $Y_0=\mu$ and for every
$r>0$,
\begin{equation}
\label{histori-snake}
\langle Y_r,\Phi\rangle
=\sum_{i\in I} \int_0^\infty \rd _s\ell^r_s(\zeta^i)\,\Phi(\omega^i_s),
\end{equation}
for every nonnegative measurable function $\Phi$ on $\mathcal{W}$. Note the relation
$\langle X_r,\varphi\rangle = \int Y_r(\rd \w)\, \varphi(\wh \w)$.

The range $\mathcal{R}^X$ is the closure in $\R$ of the set
$$\bigcup_{r\geq 0} {\rm supp}(X_r),$$
and, similarly, we define $\mathcal{R}^Y$ as the closure in $\mathcal{W}$ of
$$\bigcup_{r\geq 0} {\rm supp}(Y_r).$$
We note that
$$\mathcal{R}^X={\rm supp}(\mu) \cup\Bigg( \bigcup_{i\in I} \{\hat \omega^i_s:s\geq 0\}\Bigg)$$
and 
$$\mathcal{R}^Y={\rm supp}(\mu) \cup\Bigg( \bigcup_{i\in I} \{\omega^i_s:s\geq 0\}\Bigg).$$

We set
$$m_X:=\inf\,\mathcal{R}^X.$$
From the preceding formulas and the uniqueness of the minimizing path in the case of the Brownian snake, it 
immediately follows that there is a unique stopped path $\w_{\rm min}\in \mathcal{R}^Y$
such that $\wh \w_{\rm min}= m_X$. Our goal is to describe the distribution 
of $\w_{\rm min}$. We first observe that the distribution of $m_X$
is easy to obtain from \eqref{lawmini} and the Brownian snake representation: We have obviously $m_X\leq m$ and, for every $x<m$,
\begin{equation}
\label{super-min}
P(m_X\geq x)= \exp\Big(-\frac{3}{2}\int \frac{\mu(\rd u)}{(u-x)^2}\Big).
\end{equation}
Note that this formula is originally due to \cite[Theorem 1.3]{DIP}. It follows that
$$P(m_X=m)= \exp\Big(-\frac{3}{2}\int \frac{\mu(\rd u)}{(u-m)^2}\Big).$$
Therefore, if $\int (u-m)^{-2}\mu(\rd u)<\infty$, the event $\{m_X=m\}$ occurs with positive probability.
If this event occurs, $\w_{\rm min}$ is just the trivial path $m$ with zero lifetime. 

\begin{proposition}
\label{jointlaw}
The joint distribution of the pair $(\w_{\rm min}(0),m_X)$ is given by the formulas
$$P(\w_{\rm min}(0)\leq a,\,m_X\leq x)
=3 \int_{-\infty}^x \rd y\,\Big(\int_{[m,a]} \frac{\mu(\rd u)}{(u-y)^3}\Big)\,
\exp\Big(-\frac{3}{2}\int \frac{\mu(\rd u)}{(u-y)^2}\Big),$$
for every $a\in[m,\infty)$ and $x\in(-\infty,m)$, and
$$P(m_X=m)=P(m_X=m,\w_{\rm min}(0)=m)= \exp\Big(-\frac{3}{2}\int \frac{\mu(\rd u)}{(u-m)^2}\Big).$$
\end{proposition}

\begin{proof} Fix $a\in[m,\infty)$, and let $\mu'$, respectively $\mu''$
denote the restriction of $\mu$ to $[m,a]$, resp.~to $(a,\infty)$. Define $X'$, respectively $X''$,
by setting $X'_0=\mu'$, resp.~$X''_0=\mu''$, and restricting the sum in the 
right-hand side of \eqref{super-snake} to indices $i\in I$ such that 
$x^i\in[m,a]$, resp. $x^i\in(a,\infty)$. Define $Y'$ and $Y''$ similarly using
\eqref{histori-snake} instead of \eqref{super-snake}. Then $X'$, respectively $X''$
is a super-Brownian motion started from $\mu'$, resp.~from $\mu''$,
and $Y'$, resp.~$Y''$ is the associated historical super-Brownian motion. 
Furthermore, $(X',Y')$ and $(X'',Y'')$ are independent.

By \eqref{super-min}, the law of $m_{X'}$ has a density on $(-\infty,m)$ given by
$$f_{m_{X'}}(y)= 3\Big(\int_{[m,a]} \frac{\mu(\rd u)}{(u-y)^3}\Big)
\exp\Big(-\frac{3}{2}\int_{[m,a]} \frac{\mu(\rd u)}{(u-y)^2}\Big)\;,\quad y\in(-\infty,m).$$
On the other hand, if $x\in(-\infty,m)$,
$$P(\w_{\rm min}(0)\leq a,\,m_X\leq x)
= P(m_{X'}\leq x,\,m_{X''}> m_{X'})= \int _{-\infty}^x \rd y\,f_{m_{X'}}(y)\,P(m_{X''}> y),$$
and we get the first formula of the proposition using \eqref{super-min} again. The second
formula is obvious from the remarks preceding the proposition.
\end{proof}

Together with Proposition \ref{jointlaw}, the next corollary completely characterizes the
law of $\w_{\rm min}$. Recall that the case where $m_X=m$ is trivial, so that we 
do not consider this case in the following statement.

\begin{corollary}
\label{super-minipath}
Let $x\in(-\infty,m)$ and $a\in[m,\infty)$. Then conditionally on $m_X=x$
and $\w_{\rm min}(0)=a$, the path $\w_{\rm min}$ is distributed as the 
process $(x+R^{(3)}_t)_{0\leq t\leq T^{(3)}}$ under $P_{a-x}$. 
\end{corollary} 

\begin{proof}
On the event $\{m_X<m\}$, there is a unique index $i_{\rm min}\in I$ such that
$$m_X=\min\{\wh\omega^{i_{\rm min}}_s:s\geq 0\}.$$
Furthermore, if $s_{\rm min}$ is the unique 
instant such that $m_X= \wh\omega^{i_{\rm min}}_{s_{\rm min}}$,
we have $\w_{\rm min}=\omega^{i_{\rm min}}_{s_{\rm min}}$, and in
particular $x_{i_{\rm min}}=\w_{\rm min}(0)$. 

Standard properties of Poisson measures now imply that,
conditionally on $m_X=x$
and $\w_{\rm min}(0)=a$, $\omega^{i_{\rm min}}$ is distributed according to
$\N_a(\cdot\mid W_*=x)$. The assertions of the corollary then follow
from Theorem \ref{lawminipath}.
\end{proof}

We could also have obtained an analog of Theorem \ref{decotheo} in the superprocess
setting. The conditional distribution of the process $X$ (or of $Y$) given the minimizing
path $\w_{\rm min}$ is obtained by the sum of two contributions. The first one (present only
if $\hat\w_{\rm min}<m$) corresponds to the minimizing ``excursion'' $\omega^{i_{\rm min}}$
introduced in the previous proof, whose conditional distribution given $\w_{\rm min}$  is described by Theorem \ref{decotheo}.
The second one is just an independent super-Brownian motion $\wt X$
started from $\mu$ and conditioned on the event $m_{\wt X}\geq \hat\w_{\rm min}$. We leave the
details of the statement to the reader.

\end{document}